\documentclass[11pt]{amsart}
\usepackage[marginratio=1:1]{geometry}
\usepackage{color}
\definecolor{refs}{rgb}{0.7,0,0}
\definecolor{ext}{RGB}{112,112,112}
\definecolor{cite}{RGB}{034,113,179}
\usepackage[colorlinks=true,citecolor=cite,linkcolor=refs,urlcolor=ext,backref=page]{hyperref}

\makeatletter
\@namedef{subjclassname@2020}{\textup{2020} Mathematics Subject Classification}
\makeatother

\newtheorem{theorem}{Theorem}[section]

\newtheorem{proposition}[theorem]{Proposition}

\theoremstyle{definition}
\newtheorem{definition}[theorem]{Definition}
\newtheorem{remark}{Remark}
\newtheorem*{example}{Example}

\newcommand{\mycomment}[1]{}

\title{Rigidity theorems for cone structures}

\author[T.~Frelik]{Tymon Frelik}\address{Institute of Mathematics, Polish Academy of Sciences,  ul. \'Sniadeckich 8, 00-656 Warszawa, Poland}

\email{t.frelik@uw.edu.pl}

\author[W.~Kry\'nski]{Wojciech Kry\'nski$^*$}
\thanks{$^*$Corresponding author}\address{Institute of Mathematics, Polish Academy of Sciences, ul. \'Sniadeckich 8, 00-656 Warszawa, Poland}


\email{krynski@impan.pl}


\newcommand{\N}{\mathbb{N}}

\newcommand{\R}{\mathbb{R}}

\newcommand{\PP}{\mathbb{P}}

\newcommand{\rk}{\operatorname{rk}}

\newcommand{\spn}{\mathrm{span}}

\newcommand{\Hom}{\mathrm{Hom}}
\newcommand{\DD}{\mathcal{D}}
\newcommand{\VV}{\mathcal{V}}
\newcommand{\FF}{\mathcal{F}}

\newcommand{\WW}{\mathcal{W}}
\newcommand{\XX}{\mathcal{X}}

\newcommand{\HHH}{\mathcal{H}}
\newcommand{\CC}{\mathcal{C}}

\newcommand{\GL}{\mathrm{GL}}
\newcommand{\Set}{\mathfrak{S}}
\newcommand{\dd}{\mathrm{d}}

\begin{document}
\date{\today}
\maketitle

\begin{abstract}
Cone structures on differentiable manifolds are fields of cones in the tangent bundle. A cone structure is isotrivial, modeled on a fixed immersed submanifold of the projective space, if at each point the projectivised cone is projectively equivalent to that submanifold. We consider cone structures arising from ordinary differential equations via a canonical construction. We prove that isotrivial cone structures in this class, modeled on generic curves in the $n$-dimensional projective space or ruled surfaces in the three-dimensional projective space, are flat. We also discuss applications to causal geometries in four dimensions and dispersionless Lax systems.

\end{abstract}

\section{Introduction}

\subsection{Results.} A \emph{cone structure} on a smooth manifold $M$ is a field of cones in the tangent bundle $TM$, or equivalently, a field of immersed submanifolds in the projectivized tangent bundle $\PP TM$. A cone structure is called \emph{isotrivial}, modeled on $\CC_0 \subset \PP^{n-1}$, if for every point $x \in M$, the corresponding submanifold of the projectivized tangent space $\PP T_xM$ is projectively equivalent to $\CC_0$.

A natural source of cone structures arises in the geometric theory of ordinary differential equations (ODEs). Indeed, every system of ODEs induces, by a tautological construction recalled briefly in the next section, a cone structure on its solution space. The cone structures obtained in this way are special: the cones are necessarily ruled, and the structures satisfy certain additional integrability conditions.

We denote by $\mathfrak{V}_m^{n-1}$ the space of $m$-dimensional immersed submanifolds of $\PP^{\,n-1}$. Throughout this paper we restrict our attention to the cases $m=1$ with arbitrary $n>3$, and $m=2$ with $n=4.$ Let $\mathfrak{R}_2^3\subset\mathfrak{V}^3_2$ denote the set of ruled surfaces in $\PP^3$.

%We equip $\mathfrak{V}_m^{n-1}$ with the Whitney $C^\infty$-topology, so that convergence means smooth convergence of local embeddings on compact sets.
We equip $\mathfrak{V}_m^{n-1}$ with the Whitney $C^\infty$-topology, so that convergence means smooth convergence on compact sets of local parametrizations. The subspace $\mathfrak{R}_2^3$ inherits the topology from $\mathfrak{V}^3_2$. In this topology, subsets defined by algebraic jet conditions are closed. Our main theorem is the following.
\begin{theorem}\label{thm:main}
There exist open dense subsets $\Set_1^{n-1} \subset \mathfrak{V}^{n-1}_1$, for $n>3$, and $\Set_2^3 \subset \mathfrak{R}^3_2$ such that any isotrivial cone structure modeled on $\CC_0 \in \Set_1^{n-1}$ arising from a scalar ordinary differential equation, or on $\CC_0 \in \Set_2^3$ arising from a system of two second-order ordinary differential equations, is flat and the corresponding equation or system is linear.
\end{theorem}
Flatness here means that the structure is locally diffeomorphic to the cone structure in a linear space obtained by translating a fixed cone at the origin to any point of the space. The theorem holds in both the real and complex categories. 

\subsection{Background and motivation.} Besides the discussed relation to differential equations, cone structures naturally appear in various geometric contexts. For instance, they have been exploited in the VMRT (Varieties of Minimal Rational Tangents) program \cite{HM,H1,H2}. In fact, a version of the rigidity theorem for complete intersections has been proved in \cite{FH} (Theorem 1.7). The authors study a notion related to our ``ODE-type'' condition, that is, the so-called \emph{characteristic connection}, which we discuss in Section \ref{sec:char_con}. The proofs of our rigidity theorems use classical tools of local differential geometry. Note that in Remark 1 we provide an algebraic condition characterizing the generic set $\Set_1^{n-1}$ in the simplest case of $n=4$.

Cone structures naturally appear also in mathematical physics. Indeed, cone structures on 4-dimensional manifolds corresponding to $\mathfrak{V}^3_2$ in our notation, generalize conformal metrics and are often referred to as causal structures, see \cite{HS,M} as well as \cite{KM} and references therein. In the metric case the cones are quadratic. Moreover, in split signature, the condition that a structure arises from ordinary differential equations is equivalent to the integrability of the $\alpha$-planes associated with the metric (cf. \cite{CDT,G}). This integrability is in turn equivalent to the metric being anti-self-dual (cf. \cite{P} and \cite{MW}). By analogy, the counterparts of anti-self-dual (ASD) metrics in the setting of general causal structures are precisely the ODE-type structures covered by Theorem \ref{thm:main}. In this language, the theorem can be rephrased as stating that, \emph{for a generic cone, the associated isotrivial ASD causal space-times are necessarily flat}. This significantly restricts the class of potentially nontrivial examples of isotrivial causal structures. Note that apart from the metric case, one may consider Cayley structures \cite{KM} to obtain non-flat causal space-times.

Cone structures also arise in the theory of partial differential equations (PDEs), where they can be studied as characteristic varieties \cite{FK1,KM}. In this setting, the integrability of the equation is reflected in the integrability of the associated cone structures. More generally, one can consider a dispersionless Lax pair $(L_0(u),L_1(u))$, understood as a pair of vector fields on a manifold $M\times \PP^1$, where the additional dimension is parameterized by a spectral parameter $\lambda$, and the vector fields depend on an unknown function $u\colon M\to\R^r$. The commutativity relation $[L_0(u),L_1(u)]=0$ yields a system of PDEs for $u=(u_1,\ldots,u_r)$. For any fixed function $u$, the system defines a cone structure on $M$ by projecting onto $M$, along the fibers of $M \times \PP^1$, the tangent distribution spanned by $L_0(u)$ and $L_1(u)$. In this way, each point $x \in M$ is assigned a cone in $T_xM$, ruled by subspaces parameterized by $\lambda$. Moreover, the commutativity of $L_i(u)$ implies the integrability of the cone structure, which in the case $\dim M=4$, under a suitable non-degeneracy condition on $L_i$, is equivalent to the fact that the cone structure comes from a system of ODEs. Our result shows that for a generic pair $(L_0(u),L_1(u))$ such a structure is flat, whenever $u$ is a solution function. In terms of PDEs it means that the associated Lax system essentially admits only trivial solutions.
In fact, such systems are expected to be highly overdetermined. In a sense, our result implies that \emph{meaningful integrable systems are rare}. Non-trivial Lax systems corresponding to certain cone structures have been recently studied in \cite{K3}.
 
\subsection*{Acknowledgments.}
W.K. would like to thank Jun-Muk Hwang for posing the problem that motivated this work, as well as for many valuable discussions and suggestions.\\
This work was partially supported by the Simons Foundation grant (award no. SFI-MPS-T-Institutes-00010825) and from State Treasury funds as part of a task commissioned by the Minister of Science and Higher Education under the project “Organization of the Simons Semesters at the Banach Center - New Energies in 2026-2028” (agreement no. MNiSW/2025/DAP/491).

\section{Preliminaries}
\subsection{Cone structures.}
A cone structure $\CC$ on a manifold $M$ is a smooth field of immersed submanifolds in the projectivized tangent bundle $\PP TM$
\[
M \ni x \mapsto \CC_x \subset \PP(T_x M).
\]  
We denote by $\hat{\CC}_x$ the preimage of $\CC_x$ under the quotient map $q\colon T_x M \to \PP(T_x M)$. In other words, $\hat{\CC}_x$ is a cone in the tangent space $T_x M$ for any $x \in M$.

Assume $\dim M=n$. We say that a cone structure is \emph{isotrivial} if there exists a submanifold $\CC_0 \subset \PP^{n-1}$ such that, for every $x \in M$, $\CC_x$ is projectively equivalent to $\CC_0$. In this case, we say that $\CC$ is modeled on $\CC_0$. Two cone structures $\CC_i\subset \PP TM_i$, $i=1,2$, on manifolds $M_1,M_2$ are equivalent if there exists a diffeomorphism $\Phi\colon M_1\rightarrow M_2$ such that $\PP\Phi_*\colon\PP TM_1\to \PP TM_2$ satisfies $\PP\Phi_*(\CC_1)=\CC_2$. 

Later, we shall focus on two cases where $\CC_x$ is either a curve or a ruled surface in the 3-dimensional projective space $\PP(T_x M)$. In general, one can also consider structures such that $\CC_x$ belongs to the set $\mathfrak{R}_m(T_x M)$ of $m$-dimensional immersed submanifolds of $\PP(T_x M)$ that are ruled by a 1-parameter family of $(m-1)$-dimensional projective subspaces. Note that for $m=1$, $\mathfrak{R}_1(T_x M)$ coincides with $\mathfrak{V}_1(T_x M)$, the set of all  curves in $T_x M$.

Any cone structure such that $\CC_x \in \mathfrak{R}_m(T_x M)$ for all $x\in M$, has an equivalent description as a smooth field of immersed curves in the Grassmannian bundle $\mathrm{Gr}_m(TM)$ of $m$-planes. Indeed, the ruling of $\CC_x$ implies that $\hat{\CC}_x$ is a 1-parmeter family of $m$-dimensional subspaces of $T_x M$. We shall denote the corresponding curves in the Grassmannian bundle by $N^\CC_x$. Consequently, $N^\CC$ is a fibred bundle over $M$ with 1-dimensional fibers. Note that in the case $m=1$, $N^\CC=\CC$.

\subsection{Associated regular pairs and cone structures of equation type.}
Given a cone structure $\CC$ on $M$, such that $\CC_x\in\mathfrak{R}_m(T_xM)$ for all $x\in M$, let $\pi\colon N^\CC\to M$ denote the natural projection inherited from the Grassmannian bundle $\mathrm{Gr}_m(TM)\to M$. Assume $\dim M=n$, then $N^\CC$ is a manifold of dimension $n+1$. Using $\pi$, we associate to $\CC$ a pair $(\XX,\VV)$ of tautologically defined distributions on $N^\CC$. Firstly we define
\[\XX=\ker \pi_*.\]
Then, for $p\in N^\CC_x$ 
\[
\VV(p)=\pi^{-1}_*(p),
\]
where $p$ on the right-hand side is understood as an $m$-dimensional subspace of $T_xM$. We may equivalently define $\XX$ pointwise as $\XX(p)=\pi^{-1}_*(0)$. Thus, clearly, $\XX\subset \VV$. Note that $\rk \XX=1$ and $\rk\VV=m+1$.
\begin{definition}
Let $\CC$ be a cone structure on $M$, $\dim M=n$, such that $\CC_x\in\mathfrak{R}_m(T_xM)$ for all $x\in M$. We say that the structure is \emph{non-degenerate} if there is $k\in \N$ such that $n=(k+1)m$ and
\[
\rk \VV^i=(i+1)m+1
\]
for $i=0,\ldots,k$, where $\VV^i$ are distributions on $N^\CC$ defined inductively as
\[
\VV^{i+1}=[\XX,\VV^i],\qquad \VV^0=\VV.
\]
\end{definition}
In above, the Lie bracket of distributions is defined as
\[
[\DD_1,\DD_2](x)=\spn\{[X_1,X_2](x)\ |\ X_1\in\Gamma(\DD_1),\ X_2\in\Gamma(\DD_2)\}.
\]

It is easy to see that the pair $(\XX,\VV)$ uniquely recovers the underlying cone structure with the prescribed ruling. Indeed, $M$ is the leaf space of $\XX$, and for $x\in M$, $\hat \CC_x\subset T_xM$ is given as the union of all $\pi_*(\VV(p))$ where $p\in\pi^{-1}(x)$.

\begin{definition}
Let $\CC$ be a non-degenerate cone structure on $M$, $\dim M=n$, such that $\CC_x\in\mathfrak{R}_m(T_xM)$ for all $x\in M$. We say that $\CC$ is of \emph{equation type} if the distribution $\VV\subset TN^\CC$ is locally equivalent to the Cartan distribution on the jet space $J^{k}(\R,\R^m)$.
\end{definition}

The terminology is justified by the observation that in this case there are local coordinates on $N^C$ such that $\XX$ can be identified with the line field spanned by the total derivative vector field of the form
\[
X_F=\partial_t+\sum_{j=1}^m\left(x^1_j\partial_{x^0_j}+\cdots+x^{k}_j\partial_{x^{k-1}_j}+F_j\partial_{x^{k}_j}\right),
\]
for some function $F=(F_1,\ldots,F_m)$, where $(t,x:=(x^0_j,\ldots,x^{k}_j)\ |\ \ j=1,\ldots,m)$, are the standard coordinates on $J^{k}(\R,\R^m)$ corresponding to the consecutive derivatives of the dependent variable: $x^i_j=x^{(i)}_j=\left(\frac{\dd}{\dd t}\right)^ix_j$. In other words, $(\XX,\VV)$ locally encodes a system of ODEs of order $k$, given explicitly by
\begin{equation}\label{eq:ODE}
x^{(k+1)}=F(t,x,x',\ldots,x^{(k)}),\qquad x\in\R^m
\end{equation}
determined uniquely up to contact transformations of coordinates.

Conversely, given a system of the form \eqref{eq:ODE}, one obtains a canonical cone structure $\CC_F$ on the solution space $M_F=J^{k}(\R,\R^m)/X_F$ defined by projecting the Cartan distribution
\[
\VV_{\mathrm{Cartan}}=\spn\left\{X_F,\partial_{x^{k}_j}\ |\ j=1,\ldots,m\right\}
\]
along the integral curves of $X_F$ from $J^{k}(\R,\R^m)$ to $M_F$. Notice that in this context $N^{\CC_F}$ is identified with $J^{k}(\R,\R^m)$.

We have the following geometric characterizations of \emph{cone structures of equation type}. The first one follows from the standard description of the Cartan distribution as a (regular) Goursat distribution (see \cite{Y} for a general exposition and geometric account of the Cartan distribution on jet spaces, and also \cite{K2} for a formulation relevant in the present context). The second one follows from the description of a Cartan distribution associated with a pair of second order ODE as a \emph{path geometry}.
\begin{proposition}\label{prop1}
A non-degenerate cone structure $\CC$ on a manifold $M$ of dimension $n>2$, such that $\CC_x$ is a curve in $\PP(T_xM)$ for any $x\in M$, is of equation type if and only if the associated pair $(\XX,\VV)$ on $N^\CC$ satisfies $[\VV^i,\VV^i]=\VV^{i+1}$ for $i=0,\ldots n-2$.
\end{proposition}

Note that for $n=3$ the condition of Proposition \ref{prop1} states that $\VV$ is the Engel $(2,3,4)$-distribution, which is a generic condition. Consequently, all non-degenerate 3-dimensional cone structures are necessarily of equation type (see also \cite{FKN,HS}). In our main Theorem \ref{thm:main} it is assumed that $n>3$. 

\begin{proposition}\label{prop2}
A non-degenerate cone structure $\CC$ on a manifold $M$ of dimension~$n=4$, such that $\CC_x$ is a ruled surface in $\PP(T_xM)$ for any $x\in M$, is of equation type if and only if for the associated pair $(\XX,\VV)$ on $N^\CC$ there is an integrable subdistribution $\WW\subset \VV$ satisfying $\rk\WW=2$ and $[\XX,\WW]=TN^\CC$.
\end{proposition}
Note that with a chosen integrable $\WW\subset \VV$, the triple $(N^\CC,\WW,\XX)$ defines a three-dimensional path geometry (compare with \cite{G}), wherein, in the jet picture, $N^\mathcal{C}\cong J^1(\mathbb{R},\mathbb{R}^2)$,
\[
\WW=\spn\{\partial_{x^1_1},\partial_{x^1_2}\}\ \  \mathrm{and}
\ \  \XX=\spn\{\partial_t+x^1_1\partial_{x^0_1}+x^1_2\partial_{x^0_2}+F_1\partial_{x_1^1}+F_2\partial_{x_2^1}\}.
\]

\section{Rigidity of structures of equation type.}
We shall split the main Theorem \ref{thm:main} into two cases, corresponding respectively to curves in $\PP TM$ for manifolds of arbitrary dimension $n>3$ and to surfaces in $\PP TM$ for 4-dimensional manifolds $M$.
\subsection{Structures modeled on curves in $\PP^{n-1}$}
\begin{theorem}\label{thm1}
There exist open dense subsets $\Set_1^{n-1} \subset \mathfrak{V}^{n-1}_1$, for $n>3$, such that any isotrivial cone structure modeled on $\CC_0 \in \Set_1^{n-1}$ arising from a scalar ordinary differential equation is flat and the corresponding equation is linear.
\end{theorem}
\begin{proof}
Let $M$ be a manifold of dimension $n$. Fix an $n$-dimensional vector space $V$ and a basis $(v_0,\ldots, v_{n-1})$ in $V$. We shall identify $\PP^{n-1}$ with $\PP(V)$. Then, there exist functions $\phi^i$, $i=0,\ldots,n-1$, such that an open subset of $\CC_0\subset\PP^{n-1}$ can be parameterized in the following way
\[
\lambda\mapsto \spn\left\{\phi^0(\lambda)v_0+\cdots+\phi^{n-1}(\lambda)v_{n-1}\right\}\in \PP(V),
\]
Let $\CC$ be an isotrivial cone structure modeled on $\CC_0$. It follows from the isotriviality of $\CC$ that any point in $M$ has a neighbourhood with a local frame $(X_0,\ldots,X_{n-1})$ such that, for any $x$ in this neighbourhood, there is a linear map $f_x\colon T_xM\to V$ sending $X_i$ to $v_i$ and transforming  $\hat \CC_x$ to $\hat\CC_0$. By abuse of notation, let us still denote the pullback $f^*\phi$ as $\phi$ for brevity. We get that $\CC_x$ can be locally parameterized as
\[
(x,\lambda)\mapsto \R(\phi^0(\lambda)X_0(x)+\cdots+\phi^{n-1}(\lambda)X_{n-1}(x))\in \PP(T_xM).
\]
From now on we shall assume that $\CC_0$ has no symmetry. It follows that the frames $(X_0,\ldots,X_{n-1})$ as above are only given up to a conformal rescaling:
\[
(X_0,\ldots,X_{n-1})\mapsto \kappa (X_0,\ldots,X_{n-1}),
\]
where $\kappa\colon M\to\R$ is a positive function.

Let $c_{ij}^k$ be the structure functions of the frame $(X_0,\ldots,X_{n-1})$, i.e.,
\[
[X_i,X_j]=\sum_{k=0}^{n-1}c_{ij}^kX_k.
\]
It is a matter of computations to show that the structure functions transform as follows
\begin{equation}\label{conformal}
c_{ij}^k\mapsto\kappa c_{ij}^k+\delta_j^kX_i(\kappa)-\delta_i^kX_j(\kappa)
\end{equation}
under conformal transformations of the frame.

Our goal is to prove that $\smash{c_{ij}^k}=0$ modulo the conformal transformations \eqref{conformal}, for all cones $\CC_0$ from an open and dense subset of $\mathfrak{V}^{n-1}_1$ provided that 
\[
[\VV^i,\VV^i]=\VV^{i+1},
\]
as in Proposition \ref{prop1}. Indeed, if the structure functions $\smash{c^k_{ij}}$ of the frame vanish (modulo conformal transformations), then the frame is locally equivalent (modulo a conformal transformation) to the standard frame in the Euclidean space and consequently the structure is flat.

We shall denote
\[
V(x,\lambda)=\phi^0(\lambda)X_0(x)+\cdots+\phi^{n-1}(\lambda)X_{n-1}(x).
\]
Then,
\[
V_i=V^{(i)}=\left(\tfrac{\dd}{\dd\lambda}\right)^iV=(\phi^0)^{(i)}X_0+\cdots+(\phi^{n-1})^{(i)}X_{n-1}
\]
are consecutive derivatives of $V(x,\lambda)$ with respect to $\lambda$.

Notice that variables $(x,\lambda)$ form a local coordinate system on $N^\CC$. We can consider the vector fields $V_i(x,\lambda)$ as vector fields on $N^\CC$ and we obtain
\[
\begin{aligned}
\XX&=\spn\{\partial_\lambda\}\\
\VV^0_{(x,\lambda)}&=\spn\left\{\partial_\lambda,V(x,\lambda)\right\}\\
\VV^1_{(x,\lambda)}&=\spn\left\{\partial_\lambda,V(x,\lambda),V_1(x,\lambda)\right\}\\
&\vdots\\
\VV^{n-1}_{(x,\lambda)}&=\spn\left\{\partial_\lambda,V(x,\lambda),V_1(x,\lambda),\ldots,V_{n-1}(x,\lambda)\right\}.
\end{aligned}
\]
Consequently, the condition for the cone structure to be of equation type can be expressed as
\[
[V_i,V_j]\in\VV^{j+1}
\]
for $i,j=0,\ldots,n-2$ and $i<j$. Note that for $n=3$ this condition is void (see also the remark following Proposition \ref{prop1}). We have
\[
[V_i,V_j]=\sum_{s,t,k=0}^{n-1}(\phi^s)^{(i)}(\phi^t)^{(j)}c_{st}^kX_k.
\]
It follows from the non-degeneracy assumption that there exist matrix-valued functions $\beta=(\beta_i^j)_{i,j=0,\ldots,n-1}$ depending on the parameter $\lambda$ such that
\[
X_i=\sum_{j=0}^{n-1}\beta_{i}^jV_j.
\]
Indeed, $\smash{\beta=(\beta_i^j)_{i,j=0,\ldots,n-1}}$ is the inverse matrix to $\smash{((\phi^{i})^{(j)})_{i,j=0,\ldots,n-1}}$. We emphasize that the matrices depend on $\lambda$ but not on $x$. In particular, the entries $\smash{\beta_i^j}$, similarly to the functions $\phi^i$, depend solely on the fixed curve $\CC_0\subset \PP(V)$. We get that
\[
[V_i,V_j]=\sum_{s,t,k,l=0}^{n-1}(\phi^s)^{(i)}(\phi^t)^{(j)}\beta_k^lc_{st}^kV_l,
\]
and the condition for the structure to be of equation type reads
\begin{equation}\label{system}
\sum_{s,t,k=0}^{n-1}(\phi^s)^{(i)}(\phi^t)^{(j)}\beta_k^lc_{st}^k=0, \quad l=\max(i,j)+1,\ldots,n-1.
\end{equation}
It is a linear system for the structure functions $c_{st}^k$. Notice that under the conformal transformation \eqref{conformal} the left hand side of \eqref{system} does not change. In fact, we can consider a general transformation
\begin{equation}\label{conformal_general}
\tilde c_{ij}^k=\kappa c_{ij}^k+\delta_j^ka_i-\delta_i^ka_j
\end{equation}
for arbitrary functions $a_i$ and prove that $\tilde c_{ij}^k$ solve \eqref{system} if and only if $c_{ij}^k$ solve \eqref{system}.
Indeed, it is sufficient to substitute $c_{st}^k:=\delta_s^ka_t$ into \eqref{system}. For such a choice we get
\[
\sum_{s,t,k=0}^{n-1}(\phi^s)^{(i)}(\phi^t)^{(j)}\beta_k^lc_{st}^k:=\sum_{t=0}^{n-1}(\phi^t)^{(j)}a_t\sum_{s=0}^{n-1}(\phi^s)^{(i)}\beta^l_s=\sum_{t=0}^{n-1}(\phi^t)^{(j)}a_t\delta^l_i,
\]
which is zero for $l\neq i$, and in particular for $l>i$ as in \eqref{system} (we use above that, by definition, $\beta^l_s$ are coefficients of the inverse matrix to $(\phi^s)^{(i)}$). In the case of a conformal transformation we have $a_i=X_i(\kappa)$. Conversely, if the structure functions of a frame satisfy $c_{ij}^k=\delta_j^ka_i-\delta_i^ka_j$ then we can always locally find $\kappa$ such that $a_i=X_i(\kappa)$. This follows directly from the Jacobi identity which implies the compatibility conditions for differential equations $a_i=X_i(\kappa)$, $i=0,\ldots,n-1$.

The coefficients $c_{ij}^k$ are functions on $M$, i.e. depend on $x$ only, in contrast to the coefficients of system \eqref{system} which are functions of $\lambda$. Hence, the system \eqref{system} depends solely on the curve $\CC_0$ and, furthermore, each equation in \eqref{system} can be differentiated with respect to $\lambda$ arbitrarily many times, producing new equations that are also satisfied by $c_{ij}^k$. The resulting system can be evaluated at a fixed value $\lambda=\lambda_0$ giving a linear system for $c_{ij}^k$ with constant coefficients. The coefficients of this system depend on a high jet of the curve $\CC_0$ at $\lambda_0$ (which can, in fact, be made arbitrarily high by successive differentiation). It is expected that the system is overdetermined for generic curves. In fact, we shall show that for the generic $\CC_0$ the system admits only a trivial solution (modulo \eqref{conformal_general}).

The existence of a nontrivial solution can be expressed in terms of the rank of the system, which is given by a polynomial condition on the coefficients (given by the vanishing of certain determinants). Consequently, this defines a closed subset in the space of curves $\mathfrak{V}^{n-1}_1$. To complete the proof, it remains to show that this subset is of positive codimension. Since it is defined by polynomial conditions it suffices to show that there exists at least one curve satisfying the statement of the theorem.

Chose $n$ numbers, $m_0,m_1,\ldots,m_{n-1}\in\N$, all greater than $n$ and such that $m_i<m_j$ for $i<j$ and all $m_s+m_t-m_k$ as well as $n+m_t-m_0$ are different numbers for $k\neq t,s$. Let $\phi^0(\lambda)=\lambda^n+\lambda^{m_0}$ and $\phi^i(\lambda)=\lambda^{m_{i}}$ for $i>0$. Consider equation \eqref{system} with $l=n-1$. Since functions $\phi^i$ are polynomial in $\lambda$, the left hand side, multiplied by the determinant of matrix $((\phi^i)^{(j)})_{i,j=0,\ldots,n-1}$, which is in the denominator of all $\beta^l_k$, is a polynomial in $\lambda$ as well. Firstly, we shall compute the order of the coefficient next to $c^k_{st}$ as a polynomials in $\lambda$. Let $\mu=m_0+\cdots+m_{n-1}$. Then, after clearing the denominator, $\beta^{n-1}_k$ is a polynomial of order $\mu-m_k-\frac{1}{2}(n-2)(n-1)$, which is an order of the determinant of the matrix $((\phi^i)^{(j)})_{j=0,\ldots,n-2; i=0,\ldots,n-1; i\neq k}$. Consequently, the coefficient next to $c^k_{st}$ is of order
\[
m^k_{st}=\mu-m_k-\frac{1}{2}(n-2)(n-1)+m_s+m_t-i-j,
\]
where $i$ and $j$ are fixed for a given equation in the system \eqref{system}. Note that, due to assumptions on $m_i$'s, coefficients of $c^{k_1}_{s_1t_1}$ and $c^{k_2}_{s_2t_2}$ are of the same order only if $k_1=s_1$, $k_2=s_2$ and $t_1=t_2$ (or with the roles of the lower indices swapped). This automatically gives that \eqref{system} implies that $c^k_{st}$ vanish provided that $k\neq s$ and $k\neq t$, since the corresponding coefficients are independent polynomials in $\lambda$. The remaining coefficients are of the form $c^s_{st}$. For any fixed $t$ we get a system of equations for $c^s_{st}$ where $s=0,\ldots,n-1$, $s\neq t$, obtained by fixing a different value of $i$ in \eqref{system}. We have $n-3$ such equations for each $t$, which are independent. As a solution we get that $c^s_{st}$ are fixed up to 2-parameter freedom (for each given $t$).

One additional equation can be obtained by considering the term $\lambda^n$ in $\phi^0$. It contributes with a monomial of a minimal possible order $\mu+n-m_1-m_k-\frac{1}{2}(n-2)(n-1)$ in the formula for $\beta^{n-1}_k$, for $k>0$. Further on, this term gives a term of order $\mu+n-m_2-\frac{1}{2}(n-2)(n-1)+m_t-i-j$ in $\lambda$ standing next to an expression involving $c^s_{st}$ where $s>0$. This coefficient has to vanish provided that \eqref{system} holds. Ultimately we get, together with the previously obtained $n-3$ equations, $n-2$ independent equations for $c^s_{st}$ (for each $t$), which fix the coefficients uniquely up to conformal transformations \eqref{conformal_general}. This completes the proof.
\end{proof}

\begin{remark}
Existence of the non-trivial solution to system \eqref{system} implies that there is a polynomial relation between functions $\phi^i$ and their derivatives up to order $n$. This gives a necessary condition on $\CC_0$. It is still a challenge to provide a complete characterization of those $\CC_0$ which are not ``generic,'' i.e that do not satisfy the Theorem \ref{thm1}. For example, in the case of dimension 4 we can equivalently write the following equation
\[
\sum_{s,t=0}^3(\phi^s)(\phi^t)^{'}\left(\sum_{(i,j,k,l)=cycl(0,1,2,3)}(-1)^ic^i_{st}\Delta^{jkl}\right)=0
\]
where $\Delta^{pqr}$ is a determinant of a $3\times 3$ matrix $((\phi^s)^{(i)})_{i=0,1,2}^{s=p,q,r}$. This can be interpreted as a condition that the osculating cone of $\CC_0$ (the tangential variety) is contained in a variety defined by an equation of order 5 (provided there are nontrivial $c_{ij}^k$ solving the equation) -- it is a very restrictive condition.
\end{remark}

\begin{example}
In the case of $\mathrm{GL}(2)$-structures (see \cite{B,DT,KMet,K1}), the curves $\CC_x$ are rational normal curves. Then $\phi^i$ are polynomials in $\lambda$ of degree $n-1$ and it can be verified that the system \eqref{system} is underdetermined in this case, leading to the existence of non-flat structures of equation type. Specifically, if $M$ is of dimension 4, the system consists of $8$ equations for $24$ unknown functions $c_{st}^l$. One can check that the vanishing of \eqref{system} is equivalent to the vanishing of Bryant's torsion introduced in \cite{B}.
\end{example}

\subsection{Structures modeled on surfaces in $\PP^3$.}
\begin{theorem}\label{thm2}
There exist open dense subsets $\Set_2^3 \subset \mathfrak{R}^3_2$ such that any isotrivial cone structure modeled on $\CC_0 \in \Set_2^3$ arising from a system of two second-order ordinary differential equations is flat and the corresponding system is linear.
\end{theorem}
\begin{proof}
The proof proceeds along the lines of the above proof of Theorem \ref{thm1}, with Proposition \ref{prop2} replacing Proposition \ref{prop1}. Let $M$ be a manifold of dimension $4$. Fix a $4$-dimensional vector space $V$ and a basis $(v_1,v_2,v_3, v_4)$ in $V$. We shall identify $\PP^3$ with $\PP(V)$. Then, there exist functions $\phi^i_s$, $i=1,\ldots,4$, $s=1,2$ such that an open subset of a ruled surface  $\CC_0\subset\PP^3$ can be parameterized in the following way
\[
\lambda\mapsto \spn\left\{\phi^1_s(\lambda)v_1+\cdots+\phi^4_s(\lambda)v_4\ |\  s=1,2\right\}\in\mathrm{Gr}_2(V).
\]
Note that we can assume that $(\phi^t_s)_{s=1,2}^{t=3,4}$ is the identity matrix and consequently, the cone is defined by the four functions $(\phi^t_s)_{s,t=1,2}$. Thus, as a matrix-valued smooth function 
\[
\phi(\lambda)=\begin{pmatrix}
    \phi^1_1(\lambda) & \phi^2_1(\lambda) & 1 & 0\\
    \phi^1_2(\lambda) & \phi^2_2(\lambda) & 0 & 1
\end{pmatrix}^T, \qquad\det\left(\begin{smallmatrix}
    \phi^1_1 & \phi^2_1\\
    \phi^1_2 & \phi^2_2
\end{smallmatrix}\right)\neq 0,
\]
which follows from the standard presentation of an affine chart in a Grassmanian. 

Let $\CC$ be an isotrivial cone structure modeled on $\CC_0$. As in the previous proof, it follows from isotriviality that any point in $M$ has a neighborhood with a local frame $(X_1,\ldots,X_4)$ such that, for any $x$ in this neighborhood, there is a linear map $f_x\colon T_xM\to V$ sending $X_i(x)$ to $v_i$ and transforming  $\hat \CC_x$ to $\hat \CC_0$. By abuse of notation, we denote the pullback $f^*\phi$ as $\phi$, as before. 

We get that the subset $N^\CC$ of the Grassmann bundle $\mathrm{Gr}_2(TM)$ can be locally parameterized as
\[
(x,\lambda)\mapsto\spn\left\{\phi^1_s(\lambda)X_1+\cdots+\phi^4_s(\lambda)X_4\ |\  s=1,2\right\}\in\mathrm{Gr}_2(V)
\]

We shall assume that the surface $\CC_0$ does not admit any symmetry and then the frames $(X_1,\ldots,X_4)$ are only given up to a conformal rescaling
\[
(X_1,\ldots,X_4)\mapsto \kappa (X_1,\ldots,X_4),
\]
where $\kappa\colon M\to\R$ is a positive function. As before, the structural functions of the frame will be denoted by $c_{ij}^k$, i.e. $[X_i, X_j] = \sum_{k=1}^4 c_{ij}^k X_k$. Under a conformal rescaling, the structural functions transform according to \eqref{conformal}.

Denote
\[
\begin{aligned}
&W_i(x)=\phi^1_i(\lambda)X_1(x)+\phi^2_i(\lambda)X_2(x)+X_{i+2}(x)\\
&V_i(x)=\tfrac{\dd}{\dd\lambda} W_i(x)=(\phi^1_i)'(\lambda)X_1(x)+(\phi^2_i)'(\lambda)X_2(x),\quad i=1,2.
\end{aligned}
\]
For the cone structure to be of equation type, we first require that $(\partial_\lambda,W_1,W_2,V_1,V_2)$ is a $C^\infty(N^\CC)$-linearly independent set, which is equivalent to
\[
\det
\begin{pmatrix}
	(\phi^1_1)' & (\phi_1^2)'\\
	(\phi^1_2)' & (\phi_2^2)'
\end{pmatrix}\neq 0.
\]
However, since $\phi^i_j,\; i,j=1,2$ parametrize a generic surface, the (Whitney) open condition $\smash{\det\big((\phi^i_j)'_{i=1,2}\big)\neq 0}$ on the first jets holds generically. In other words, the requirement that $\smash{(\partial_\lambda,W_i,V_i)_{i=1,2}}$ is a local frame on $N^\mathcal{C}$ is satisfied for a generic $\mathcal{C}_0$. 

Secondly, following Proposition \ref{prop2}, we require that $\VV$ contains an integrable rank two sub-distribution $\WW$ such that $\VV=\XX\oplus\WW$. We find $\WW$ explicitly by expressing it in terms of two vector fields 
\[
\begin{aligned}
&L_1=W_1+\eta_1(x,\lambda)\partial_\lambda=\phi^1_1(\lambda)X_1(x)+\phi^2_1(\lambda)X_2(x)+X_3(x)+\eta_1(x,\lambda)\partial_\lambda,\\
&L_2=W_2+\eta_2(x,\lambda)\partial_\lambda=\phi^1_2(\lambda)X_1(x)+\phi^2_2(\lambda)X_2(x)+X_4(x)+\eta_2(x,\lambda)\partial_\lambda
\end{aligned}
\]
satisfying
\begin{equation}\label{eq:firstintegrability}
[L_1,L_2]=gL_1+hL_2=gW_1+hW_2+(g\eta_1+h\eta_2)\partial_\lambda,
\end{equation}
for some, $g,h\in C^\infty(N^\CC)$.
In general, we have
\[
[L_1,L_2]=[W_1,W_2]+\eta_1\partial_\lambda W_2-\eta_2\partial_\lambda W_1+\big(W_1(\eta_2)-W_2(\eta_1)+(\eta_1\eta_2'-\eta_2\eta_1')\big)\partial_\lambda,
\]
%Thus, in our problems, we seek frames $(X_1,X_2,X_3,X_4)$, whose structure functions $\smash{c^k_{ij}(x)}$ satisfy the integrability condition \ref{eq:firstintegrability}. We will show that $\smash{c^k_{ij}(x)}$ vanish identically.
and we first focus on \eqref{eq:firstintegrability} modulo $\partial_\lambda$, that is, without the vertical terms. This corresponds to the assertion that $[\WW,\WW]\subseteq \VV$. These equations can be always solved for any generic rank 3 distribution (i.e., a distribution with growth vector $(3,5)$), since every such distribution $\VV$ admits a unique rank 2 subdistribution $\WW$ defined by $[\WW,\WW]\subseteq \VV$ \cite[Lemma 6.12]{Mon}. This part of the integrability condition takes the form
\begin{equation}\label{eq:laxintegrability}
[W_1,W_2]+\eta_1\partial_\lambda W_2-\eta_2\partial_\lambda W_1
= g W_1 + h W_2,
\end{equation}
giving an explicit linear algebraic system for $\eta_1$, $\eta_2$, as well as $g$ and $h$, which can be solved explicitly. In fact, computing the bracket $[W_1,W_2]$, we obtain
\[
[W_1,W_2]=\sum_{k=1}^4\Phi^k X_k,
\]
where
\[
\Phi^k=\phi_1^i \phi_2^j c^k_{ij}+\phi_1^i c^k_{i4}+\phi_2^j c^k_{3j}+ c^k_{34}
\]
and substituting this into \eqref{eq:laxintegrability}, we find that $\eta_1$ and $\eta_2$ are given by
\[
\begin{aligned}
&\eta_1=\tfrac{1}{(\phi^1_1)'(\phi^2_2)'-(\phi^1_2)'(\phi^2_1)'}\big( 
(\phi^2_1)'\big(\Phi^1-g\phi^1_1-h\phi^1_2\big)
-(\phi^1_1)'\big(\Phi^2-g\phi^2_1-h\phi^2_2\big)
\big)\\
&\eta_2=\tfrac{1}{(\phi^1_1)'(\phi^2_2)'-(\phi^1_2)'(\phi^2_1)'}\big(
(\phi^2_2)'\big(\Phi^1-g\phi^1_1-h\phi^1_2\big)
-(\phi^1_2)'\big(\Phi^2-g\phi^2_1-h\phi^2_2\big)
\big),
\end{aligned}
\]
where, additionally $g$ and $h$ are explicitly given as
\[
g=\Phi^3,\qquad h=\Phi^4.
\]

Now, we will consider equation \eqref{eq:firstintegrability} with the vertical terms, which is the actual integrability condition, that is, $[\WW,\WW]\subset \WW$. We have the single PDE for $\eta_1,\eta_2$ arrising from coefficients multiplying $\partial_\lambda$:
\begin{equation}\label{eq:integrability2}
\eta_1\eta_2'-\eta_2\eta_1'-g\eta_1-h\eta_2+W_1(\eta_2)-W_2(\eta_1)=0.
\end{equation}

Crucially, equation \eqref{eq:integrability2} still has $\lambda$-dependent and $x$-dependent terms separated. Furthermore, it is linear in the $x$-derivatives $X_s(c^k_{ij})$, $s=1,2,3,4$ from the $W_i(\eta_j)$ terms and quadratic in $c^k_{ij}$ from the remaining terms. It does not include linear terms $c^k_{ij}$.  Thus, in total, the integrability equation \eqref{eq:integrability2} is formulated on the $d$-dimensional vector subspace
\[
C^\infty(\R)\otimes\spn\{c^k_{ij}c^r_{pq}, X_s(c^k_{ij})\}\subset C^\infty(N^\mathcal{C}),
\]
where $i,j,k,p,q,r,s=1,2,3,4$ and $d=|\{c^k_{ij}c^r_{pq}, X_s(c^k_{ij})\}|$.

Note that consecutive $\lambda$-differentiations of \eqref{eq:integrability2} produce new non-trivial equations satisfied by $\{c^k_{ij}c^r_{pq}, X_s(c^k_{ij})\}$. This yields a finite overdetermined set of arbitrary many equations. 
Schematically, equation \eqref{eq:integrability2} has the form 
\begin{equation}\label{eq:schematicform}
B_{kr}^{ijpq}(\lambda)c^k_{ij}c^r_{pq}(x)+C^{sij}_{k}(\lambda)X_s(c^k_{ij})(x)=0
\end{equation}
and the linear system obtained from the jet prolongations takes the form
\begin{align}\label{eq:jetprolongedsystem}
\begin{pmatrix}
    B & C\\
    B^{(1)} & C^{(1)}\\
    \vdots & \vdots\\
    B^{(\nu)} & C^{(\nu)}
\end{pmatrix}\begin{pmatrix}
    c^k_{ij}c^r_{pq}\\
    X_s(c^k_{ij})
\end{pmatrix}=0,\end{align}
where all coefficients $B^{(i)}$ and $C^{(i)}$ depend on the $(i+1)$-jet of the original surface in $\PP^3$, as they are expressed in terms of derivatives of the defining function $\phi(\lambda)$ up to sufficiently high order. Note that the rows of the matrix are indexed by consecutive jets. To complete the proof, it remains to show that, for a generic surface, the system admits only the trivial solution $c^k_{ij}$ modulo conformal rescaling. As in the previous case, we will show that there exists at least one ruled surface satisfying the statement of the theorem.

We construct, analogously to the earlier case, a $\mathrm{GL}(2)$-valued function
\begin{align*}
\begin{pmatrix}
    \phi^1_1 & \phi_1^2\\
    \phi^1_2 & \phi^2_2
\end{pmatrix}=\sum_{i=1}^\mu\begin{pmatrix}
    \lambda^{m_1^i} & \lambda^{m_2^i}\\
    \lambda^{m_3^i} & \lambda^{m_4^i}
\end{pmatrix}\end{align*}
for a sufficiently large $\mu\in\N$ and certain $m_a^i\in\N$ such that $m_a^i\gg m_b^{i-1}$ and which produce pairwise distinct combinations of the form $\sum_{j=1}^4c^jm_j^i$ for each value of $i=1,\ldots,\mu$, where $c^j\in\{0,1,2,3,4\}$.

After clearing the denominator of equation \eqref{eq:integrability2}, one obtains a polynomial expression in $\lambda$ whose coefficients are combinations of $c^k_{ij}c^p_{qr}$ and  $X_s(c^k_{ij})$ multiplied by monomials in $m^i_j$'s. The exponents of $\lambda$ involve sums of $m^i_j$ with, in general, arbitrary indices $i$ and $j$ in the range $1,\ldots,\mu$ and $1,\ldots,4$, respectively. However, for each fixed $i$, there is a group of exponents involving only $m^i_1,\ldots, m^i_4$. The expressions corresponding to these powers of $\lambda$ are identical for different $i$, differing only by the respective $m^i_j$ appearing in the formulas. For instance, next to $\lambda^{3m^i_2+3m^i_3}$ we have
\[
\begin{aligned}
&m^i_2(m^i_3)^2 \Big(c^1_{12}c^2_{12}+ 2\,c^1_{12}c^3_{31}+ c^2_{12}c^4_{24}+ 2\,c^3_{31}c^4_{24}- X_1(c^1_{12})-X_1(c^4_{24})\Big)\\
&\qquad\qquad
- (m^i_2)^2m^i_3 \Big(c^1_{12}c^2_{12}+ c^1_{12}c^3_{31}+ 2\,c^2_{12}c^4_{24}+ 2\,c^3_{31}c^4_{24}+ X_2(c^2_{12})+ X_2(c^3_{31})\Big).
\end{aligned}
\]
Taking the expressions corresponding to different $i$, i.e., with different values of the $m^i_j$, we conclude that, in order for \eqref{eq:integrability2} to hold, each coefficient corresponding to a monomial in $m^i_j$'s must vanish. In particular, from the above expression we obtain two equations
\[
\begin{aligned}
&c^1_{12}c^2_{12}+ 2\,c^1_{12}c^3_{31}+ c^2_{12}c^4_{24}+ 2\,c^3_{31}c^4_{24}- X_1(c^1_{12})- X_1(c^4_{24})=0,\\
&c^1_{12}c^2_{12}+ c^1_{12}c^3_{31}+ 2\,c^2_{12}c^4_{24}+ 2\,c^3_{31}c^4_{24}+ X_2(c^2_{12})+ X_2(c^3_{31})=0.
\end{aligned}
\]
In general, we consider equations obtained by requiring the coefficients corresponding to monomials of the form
\[
(m^i_1)^{p_1}(m^i_2)^{p_2}(m^i_3)^{p_3}(m^i_4)^{p_4}\lambda^{c^1m^i_1+\cdots+c^4m^i_4}
\]
to vanish, for all possible $p_1,\ldots,p_4, c^1,\ldots,c^4 \in \N$. The constants $c^j$ are never grater than $4$ and our assumption ensures that all exponents $\sum_{j=1}^4c^jm_j^i$ are different. To proceed, we have written a simple code in Wolfram Mathematica for the purpose of analyzing equation \eqref{eq:integrability2}. 

At first we check that 12 quadratic terms of the form
\[
(c^{k_0}_{i_0j_0})^2, \qquad i_0\neq k_0, \quad j_0\neq k_0
\]
correspond to distinct powers of $\lambda$. Hence, they all vanish if  \eqref{eq:integrability2} holds. Explicitly, the terms we eliminate, together with their unique degrees (for any $i=1,\ldots,\mu$), are
\renewcommand{\arraystretch}{1.3}
\[
\begin{array}{ccl|ccl}
(c^1_{24})^2 & \text{deg.} & 3m^i_3+m^i_4 & (c^1_{32})^2 & \text{deg.} & m^i_3+3m^i_4 \\
(c^1_{34})^2 & \text{deg.} & m^i_3+m^i_4 & (c^2_{14})^2 & \text{deg.} & 3m^i_1+m^i_2 \\
(c^2_{31})^2 & \text{deg.} & m^i_1+3m^i_2 & (c^2_{34})^2 & \text{deg.} & m^i_1+m^i_2 \\
(c^3_{12})^2 & \text{deg.} & 4m^i_1+m^i_3+3m^i_4 & (c^3_{14})^2 & \text{deg.} & 4m^i_1+m^i_3+m^i_4 \\
(c^3_{24})^2 & \text{deg.} & 2m^i_1+3m^i_3+m^i_4 & (c^4_{12})^2 & \text{deg.} & 3m^i_1+m^i_2+4m^i_4 \\
(c^4_{31})^2 & \text{deg.} & m^i_1+3m^i_2+2m^i_4 & (c^4_{32})^2 & \text{deg.} & m^i_1+m^i_2+4m^i_4 \\
\end{array}
\]
Furthermore, we may eliminate terms of the form $c^{k_0}_{i_0j_0}c^k_{ij}$, with $i_0\neq k_0\neq j_0$, for arbitrary $i,j,k$. For the remaining coefficients $c_{ij}^k$, we are left with equations involving their derivatives with respect to $X_1,\ldots,X_4$ as well as certain relations that do involve only the quadratic terms. Direct, but long computations, shows that these equations, together with the Jacobi identities, reduce the equations involving the derivatives to equations of the form
\begin{equation}\label{eq:der}
X_i(c^p_{pj})-X_j(c^q_{qi})+c^p_{pj}c^j_{ij}+c^q_{qi}c^i_{ij}=0.
\end{equation}
On the other hand, we also get formulae of the form
\begin{equation}\label{eq:quadr}
-2c^p_{pi}c^q_{qi}+\left(c^p_{pi}\right)^2+\left(c^q_{qi}\right)^2=0,
\end{equation}
for $p\neq i$ and $q\neq i$; for instance, the expression $(c^2_{12}+c^3_{31})^2=0$ can be found next to exponent $m_1^i+3m_2^i+2m_3^i$ or $(c^2_{12}+c^4_{41})^2=0$ next to exponent $3m_1^i+m_2^i+2m_4^i$.
%For instance, from the above equations obtained from the coefficient of  $\lambda^{3m^i_2+3m^i_3}$ and form the Jacobi identity we get
%\[
%X_1(c^1_{12})+X_2(c^4_{14})=c^1_{12}c^2_{12}-c^1_{12}c^4_{14}
%\]
%among others relations.
From \eqref{eq:quadr} we get that $c^p_{pi}=c^q_{qi}$. From \eqref{eq:der}, we directly deduce that all $c^p_{pj}$ are necessarily of the form $X_j(\kappa)$ for some fixed function $\kappa$, i.e., the only freedom in choosing $c^k_{ij}$ comes from conformal transformations of the frame. This follows from the Poincaré lemma applied to the $1$-forms
\[
\alpha:=c^{p_1}_{{p_1}1}\theta^1+c^{p_2}_{{p_2}2}\theta^2+c^{p_3}_{{p_3}3}\theta^3+c^{p_4}_{{p_4}4}\theta^4,
\]
where $\theta^i$ are 1-forms dual to $X_i$ and $p_i$ are indices such that $p_i\neq i$. Indeed, the $1$-forms are necessarily closed, provided equation \eqref{eq:der} holds. This completes the proof.
\end{proof}

\begin{example}
Surfaces (apart from conformal structures) admitting non-flat isotrivial cone structures include the Cayley structures described in \cite{KM}, as well as the structures presented in \cite{K3}, modeled on certain homogeneous surfaces (see \cite{DSV,DDKR} for the lists of the homogeneous surfaces in $\PP^3$).
\end{example}

\begin{remark}
We expect that our main theorem holds for any $m$ and $n>3$. In fact, the line of proof presented here in the case $m=1$, as well as for $m=2$ with $n=4$, seems to be valid in a more general context as well, but the details (in particular, the construction of an example needed in the proof) should be carefully checked. In any case, the cases considered in the present paper seem to be the most interesting.
\end{remark}

\paragraph{\bf Lax systems.}
Once a coordinate system on $M$ is fixed, the vector fields $L_1, L_2$ can be interpreted as a dispersionless Lax pair defining a system of equations for the coefficients of the vector fields $X_i$, which provide an adapted frame for the underlying cone structure. This yields a master equation governing all isotrivial cone structures modeled on a fixed cone. However, one may also assume a more specific ansatz for the vector fields $X_i$, depending on a finite number of unknown functions and their derivatives (see, e.g., \cite[Section 8]{FK1} for various classical equations and their Lax pairs arising in conformal geometry). An upshot of the theorem is that, for a generic surface in $\PP^3$, the corresponding Lax systems generating isotrivial cone structures modeled on this surface, as described in the Introduction, necessarily admit only the trivial solution. This supports the statement that \emph{meaningful integrable systems are rare} and, in particular, it is a hard task to find Lax systems depending on the spectral parameter in a non-polynomial way.

\subsection{Intrinsic torsion.}
In this section we provide further characterization of the set of generic submanifolds in Theorems \ref{thm1} and \ref{thm2} in terms of the naturally associated torsion.
%We restrict ourselves to the case of curves, i.e.~$m=1$.
For a detailed study of connections associated with the cone structures, see~\cite{H3}. Here, we recall only the constructions relevant to our purposes.
Firstly, we restrict ourselves to the case of curves, i.e.~$m=1$.

For each $p \in N^\CC$, there is a natural filtration of $T_pN^\CC$ defined by distributions $\VV^i$. Indeed, we have
\[
\XX\subset\VV^0\subset\ldots\subset\VV^{n-1}=TN^\CC.
\]
This filtration reduces the full frame bundle of $N^\CC$ to a subbundle $F_\CC$ with the structure group $G < \mathrm{GL}(n+1)$, consisting of upper triangular matrices.  Indeed, we shall consider frames
\[
(e_0,e_1,\ldots,e_n)
\]
on $N^\CC$ such that $e_0$ spans $\XX$, and $e_0,\ldots,e_{i+1}$ span $\VV^i$ for $i=0,\ldots,n-1$. There is clearly a natural action of upper triangular matrices on this set of frames. Let $\mathfrak{g}$ denote the associated Lie algebra. We shall consider principal $G$-connections on $N^\CC$. These are defined as $G$-equivariant 1-forms $\omega$ on $F_\CC$ with values in $\mathfrak{g}$, such that the evaluation of $\omega$ on any fundamental vector field on $F_\CC$ equals the corresponding element of $\mathfrak{g}$.

It follows from the definition that the difference of two principal $G$-connections is a horizontal 1-form (a pullback of a 1-form on $\CC$). For a given frame $(e_0,e_1,\ldots,e_n)$ at $p\in\CC$, the tautological soldering form identifies $T_p\CC$ with $\R^{n+1}$ by sending $(e_i)$ to the standard basis of $\R^{n+1}$. Under this identification the difference of two $G$-connections is interpreted as a section of $\Hom(\R^{n+1},\mathfrak{g})$ over $\CC$.

Recall that the Spencer operator
\[
\delta\colon\Hom(\R^{n+1},\mathfrak{g})\to\Hom(\R^{n+1}\wedge\R^{n+1},\R^{n+1})
\]
is given by
\[
\delta(A)(v,w)=A(v)w-A(w)v.
\]
It computes the difference of torsion tensors corresponding to two principal $G$-connections. Then, the essential torsion of a $G$-structure is defined as an element of
\[
\Hom(\R^{n+1}\wedge\R^{n+1},\R^{n+1})/\operatorname{Im}(\delta)
\]
given by
\[
T(\nabla)\mod\operatorname{Im}(\delta),
\]
where $\nabla$ is an arbitrary $G$-connection. It can be understood as the part of the torsion tensor which does not depend on a particular choice of $\nabla$.

In the specific case of $G$-structures on $\CC$ we observe that
\[
\delta(A)(e_i,e_j)\in\spn\left\{e_0,e_1,\ldots,e_{\max(i,j)}\right\}.
\]
Defining $T_{ij}^k$ by the formula
\[
T(\nabla)(e_i,e_j)=\sum_kT_{ij}^ke_k
\]
we get that the components of the essential torsion are given by $T_{ij}^k$ such that $k>\max(i,j)$.

Consequently, Proposition \ref{prop1} can be equivalently phrased as follows:
\begin{proposition}\label{prop3}
A non-degenerate cone structure structure is of equation type if and only if the only non-trivial coefficients $T_{ij}^k$ of the essential torsion are such that $k=\max(i,j)+1$.
\end{proposition}
\begin{proof}
Let $\nabla$ be a $G$-connection on $\CC$. It can be equivalently interpreted as a linear connection on $\CC$ whose parallel transport preserves filtration $\VV^i$. The torsion coefficients can be computed using the frame $e_0:=\partial_\lambda$ and $e_i:=V_{i-1}$ where $V_0,\ldots,V_{n-1}$ are vector fields on $\CC$ defined in the previous section, which are generators of the distributions $\VV^i$.  In general we have
\[
T(\nabla)(e_i,e_j)=\nabla_{e_i}e_j-\nabla_{e_j}e_i-[e_i,e_j],
\]
and since $\nabla_{e_i}e_j\in\VV^j$ and $\nabla_{e_j}e_i\in\VV^i$ we get that Proposition \ref{prop1} implies that all $T_{ij}^k$ vanish for $k>\max(i,j)+1$ in the case of $\CC$ of equation type.
\end{proof}

In the case of surfaces in $\PP^3$, one has the filtration 
\[
\XX\subset \VV=\XX\oplus\WW\subset TN^\CC
\]
together with the additional condition that the integrable subdistribudion $\WW\subset \VV$ satisfies $[\XX,\WW]=TN^\CC$. Therefore, adapted frames $(e_0,e_1,e_2,e_3,e_4)$ are such that:
\[
\XX=\mathrm{span}\{e_0\},\quad \WW=\mathrm{span}\{e_1,e_2\},\quad \text{and}\quad [e_1,e_2]=0\;\;\mathrm{mod}\;\;\WW.
\]
The subgroup 
\[
G=\left\{\begin{pmatrix}
    a & 0 & v^T\\
   0 & B & C \\ 
   0 & 0 & D
\end{pmatrix}\in \GL(5)\quad \Bigg| \quad a\in\R\setminus\{0\},\; B,D\in\mathrm{GL}(2),\; C\in\R^{2\times 2},\; v\in \R^2\right\}
\]
preserves such adapted frames. Hence, it is the structure group of the reduced frame bundle $F_\CC\rightarrow N^\CC$. The essential torsion in this case is 7-dimensional. It decomposes into three parts
\[
\Hom\left(\WW^{\wedge 2}, TN^\CC/\VV\right)
\;\oplus\;
\Hom\left(\WW\wedge (TN^\CC/\VV),\XX\right)
\;\oplus\;
\Hom\left((TN^\CC/\VV)^{\wedge 2}, \XX\right)
\]

Clearly, in order for the structure to be of equation type, the first term needs to vanish. Analogously to the previous case, Proposition \ref{prop2} may be equivalently phrased as
\begin{proposition}
 A non-degenerate cone structure is of equation type if and only if the essential torsion of the associated $G$-structure satisfies:
\begin{enumerate}
\item $T^k_{12}=0, \quad k\in\{0,3,4\}$ and
\item the $2\times 2$ matrix cell $T^j_{0i},\; i\in \{1,2\}$, $j\in\{3,4\}$ is invertible.
\end{enumerate}
\end{proposition}

\subsection{Characteristic connection.}\label{sec:char_con}
Consider an isotrivial cone structure $\CC$ modeled on a ruled surface $\CC_0 \in \mathfrak{R}^3_2$. Following \cite{H0}, we consider the so-called \emph{characteristic connections} for $\CC$, whose existence implies that the structure is flat \cite[Theorem 5.8]{H0}. Our aim here is to determine, \emph{a posteriori}, the characteristic connection in the context of Theorem~\ref{thm2}. Unlike before, we now work locally on the 6-dimensional manifold $\CC \subset \PP(TM)$ rather than on the 5-dimensional $N^\CC \subset \mathrm{Gr}_2(TM)$. Let $\pi \colon \CC \to M$ be the projection.

Similarly to the pair $(\XX,\VV)$ on $N^\CC$, we define a pair $(\DD,\HHH)$ on $\CC$. Namely, $\DD$ is the vertical rank-2 distribution on $\CC$,
\[
\DD = \ker \pi_*,
\]
while $\HHH$ is the tautological rank-3 distribution on $\CC$ inherited from the structure of $\PP(TM)$, i.e., for $p \in \CC$,
\[
\HHH(p) = \pi_*^{-1}(p).
\]
Clearly, $\DD \subset \HHH$ and the quotient $\HHH/\DD$ is isomorphic to the tautological line bundle over $\CC$. A line subbundle $\FF \subset \HHH$ is called a \emph{characteristic connection} of $\CC$ if
\[
\HHH = \FF \oplus \DD
\]
and
\[
[\FF,[\FF,\HHH]] \subset [\FF,\HHH].
\]
Note that this condition means that $\FF$ is contained in the Cauchy characteristic of $[\FF,\HHH]$.

To pass from the pair $(\XX,\VV)$ on $N^\CC$ to the pair $(\DD,\HHH)$ on $\CC$, we note that $\CC$ can be naturally identified with $\PP(\VV/\XX)$. In particular, the local frame $X_1,\ldots,X_4,\partial_\lambda$ used in the proof of Theorem~\ref{thm2} on $N^\CC$ can be extended by $\partial_v$, where $v$ parameterizes the fibers of $\PP(\VV/\XX)$ over $N^\CC$. Explicitly, $\VV/\XX$ is identified with $\WW\subset \VV$, and it has a fixed basis $(L_1,L_2)$ which descents to an affine parameter $v$ on fibers of $\PP(\VV/\XX)$ defined by formula $v\mapsto \spn\{L_1+vL_2\}$. In this setting,
\[
\HHH = \spn\{\partial_\lambda,\partial_v,L_1 + vL_2\}.
\]
Now, if the structure satisfies the assertion of the theorem (i.e.\ $\CC_0$ lies in the generic open dense subset), then $L_i$, $i=1,2$, have no vertical part $\partial_\lambda$, and all vector fields in the frame $(X_1,\ldots,X_4)$ on $M$ commute. It follows that the vector field $L_C = L_1 + vL_2$ on $\CC$ spans a line bundle satisfying the condition for a characteristic connection.

Indeed, for $\FF = \spn\{L_C\}$, one has
\[
[\FF,\HHH] = \spn\{\partial_\lambda,\partial_v,L_1,L_2,L_1' + vL_2'\},
\]
where, as before, $L_i' = \partial_\lambda L_i$. We observe that $L_C$ commutes with all $X_i$ and, since it does not contain $\partial_\lambda$, also with $L_i$ and $L_i'$. The only nontrivial brackets are of the form
\[
[\partial_\lambda,L_C] = L_1' + vL_2',\quad [\partial_v,L_C]=L_2,
\]
and hence $[\FF,[\FF,\HHH]] \subset [\FF,\HHH]$. In this way the characteristic connection is recovered from the Lax pair.

\end{document}